\documentclass[11pt, reqno]{amsart}

\usepackage{latexsym}
\usepackage{amssymb}
\usepackage{mathrsfs}
\usepackage{amsmath}
\usepackage{fancybox,color}
\usepackage{enumerate}
\usepackage[latin1]{inputenc}
\usepackage{graphicx}

\usepackage{xcolor}
\usepackage{enumitem}
\usepackage[active]{srcltx}
\usepackage[pagebackref,colorlinks,citecolor=blue,linkcolor=blue]{hyperref}
\usepackage{eurosym}
\usepackage{url}

 \def\1{\raisebox{2pt}{\rm{$\chi$}}}

\newtheorem{theorem}{Theorem}[section]
\newtheorem{corollary}[theorem]{Corollary}
\newtheorem{lemma}[theorem]{Lemma}
\newtheorem{proposition}[theorem]{Proposition}
\newtheorem{definition}[theorem]{Definition}

\newcommand{\R}{{\mathbb R}}

\newcommand{\N}{{\mathbb N}}

\renewcommand{\L}{\mathcal{L}}

 \newcommand{\eps}{{\varepsilon}}
 \def\1{\raisebox{2pt}{\rm{$\chi$}}}
 
\newcommand{\Lip}{\operatorname{Lip}}
\newcommand{\abs}[1]{\left|#1\right|}

\def\vint_#1{\mathchoice%
          {\mathop{\kern 0.2em\vrule width 0.6em height 0.69678ex depth -0.58065ex
                  \kern -0.8em \intop}\nolimits_{\kern -0.4em#1}}%
          {\mathop{\kern 0.1em\vrule width 0.5em height 0.69678ex depth -0.60387ex
                  \kern -0.6em \intop}\nolimits_{#1}}%
          {\mathop{\kern 0.1em\vrule width 0.5em height 0.69678ex
              depth -0.60387ex
                  \kern -0.6em \intop}\nolimits_{#1}}%
          {\mathop{\kern 0.1em\vrule width 0.5em height 0.69678ex depth -0.60387ex
                  \kern -0.6em \intop}\nolimits_{#1}}}
\def\vintslides_#1{\mathchoice%
          {\mathop{\kern 0.1em\vrule width 0.5em height 0.697ex depth -0.581ex
                  \kern -0.6em \intop}\nolimits_{\kern -0.4em#1}}%
          {\mathop{\kern 0.1em\vrule width 0.3em height 0.697ex depth -0.604ex
                  \kern -0.4em \intop}\nolimits_{#1}}%
          {\mathop{\kern 0.1em\vrule width 0.3em height 0.697ex depth -0.604ex
                  \kern -0.4em \intop}\nolimits_{#1}}%
          {\mathop{\kern 0.1em\vrule width 0.3em height 0.697ex depth -0.604ex
                  \kern -0.4em \intop}\nolimits_{#1}}}

\newcommand{\aveint}[2]{\mathchoice%
          {\mathop{\kern 0.2em\vrule width 0.6em height 0.69678ex depth -0.58065ex
                  \kern -0.8em \intop}\nolimits_{\kern -0.45em#1}^{#2}}%
          {\mathop{\kern 0.1em\vrule width 0.5em height 0.69678ex depth -0.60387ex
                  \kern -0.6em \intop}\nolimits_{#1}^{#2}}%
          {\mathop{\kern 0.1em\vrule width 0.5em height 0.69678ex depth -0.60387ex
                  \kern -0.6em \intop}\nolimits_{#1}^{#2}}%
          {\mathop{\kern 0.1em\vrule width 0.5em height 0.69678ex depth -0.60387ex
                  \kern -0.6em \intop}\nolimits_{#1}^{#2}}}

\newcommand{\Om}{\Omega}

\newcommand{\dist}{\operatorname{dist}}

\usepackage{upref}

\makeatletter\def\SL@eqnlefttext #1{\hbox to 0pt{\kern 75 pt 
\llap{\SL@margintext{#1}\quad}\hss}}
\makeatother

\begin{document}

\title[Asymptotic $C^{1,\gamma}$-regularity for values]{Asymptotic $C^{1,\gamma}$-regularity for value functions to uniformly elliptic dynamic programming principles}


\author[Blanc]{Pablo Blanc}
\address{Pablo Blanc \hfill\break\indent
        Departamento  de Matem{\'a}tica, FCEyN,
        Universidad de Buenos Aires,
        Pabe-llon I,
         Ciudad Universitaria (C1428BCW),
        Buenos Aires, Argentina.}
\email{pblanc@dm.uba.ar}

\author[Parviainen]{Mikko Parviainen}
\address{Mikko Parviainen \hfill\break\indent
Department of Mathematics and Statistics, University of
Jyv\"askyl\"a, PO~Box~35, FI-40014 Jyv\"askyl\"a, Finland.}
\email{mikko.j.parviainen@jyu.fi}

   \author[Rossi]{Julio D. Rossi} \address{Julio D. Rossi\hfill\break\indent
        Departamento  de Matem{\'a}tica, FCEyN,
        Universidad de Buenos Aires,
        Pabe-llon I,
         Ciudad Universitaria (C1428BCW),
        Buenos Aires, Argentina.}
\email{jrossi@dm.uba.ar}


\date{\today}

\begin{abstract}
In this paper we prove an asymptotic $C^{1,\gamma}$-estimate for value functions of stochastic processes related to uniformly elliptic dynamic programming principles.  As an application, this allows us to pass to the limit with a discrete gradient and then to obtain a $C^{1,\gamma}$-result for the corresponding limit PDE.
\end{abstract}

\maketitle


\section{Introduction}
\label{sec:intro}

In this paper we deal with regularity estimates for solutions to certain 
equations that appear naturally when one considers value functions of 
stochastic games. 
We consider a class of discrete stochastic processes or equivalently functions satisfying a \textit{dynamic programming principle (DPP)}.
To be more precise, for example, we can consider functions that satisfy a uniformly elliptic DPP such as 
\[
u(x)=\alpha \frac{u(x+\eps \nu)+u(x-\eps \nu)}{2}
+\beta\vint_{B_1} u(x+\eps y)\,dy
\]
where $\nu\in B_\Lambda$ is a fixed vector, $B_{\Lambda}$ is a ball in $\R^n$ of radius $\Lambda$ and the parameters that appear in the equation verify  $\beta,\eps,\Lambda>0$, $\alpha\geq 0$, and  $\alpha+\beta=1$. 
Our result also applies to solutions of the nonlinear DPP given by
\[
u(x)=\alpha \sup_{\nu\in B_\Lambda} \frac{u(x+\eps \nu)+u(x-\eps \nu)}{2}
+\beta\vint_{B_1} u(x+\eps y)\,dy.
\]
In \cite{brustadlm20} a control problem associated to the nonlinear example is presented and, in the limit as $\eps \to 0$, a local PDE involving the dominative $p$-Laplacian operator arises.
Heuristically, the above DPP can be understood by considering a value $u$ at $x$, which can be computed by summing up different outcomes with corresponding probabilities: either a maximizing controller who gets to choose $\nu$ wins (probability $\alpha$), or a random step occurs (with probability $\beta$) within a ball of radius $\varepsilon$. If the controller wins, the position moves to $x+\eps \nu$ (probability $1/2$) or to $x-\eps \nu$ (probability $1/2$). 
 Also Isaacs type dynamic programming principle 
\[
u(x)=\alpha \sup_{V \in \mathcal V}\inf_{\nu \in V}  \frac{u(x+\eps \nu)+u(x-\eps \nu)}{2}
+\beta\vint_{B_1} u(x+\eps y)\,dy,
\]
with  $\mathcal V \subset\mathcal P(B_\Lambda)$, a subset of the power set, and $\beta>0$ can be mentioned as an example.

Our main contribution is a $C^{1,\gamma}$-estimate in Theorem \ref{teo.C1a} and with $x$-dependent Lipschitz running cost in Theorem \ref{teo.C1a.f}. 
To obtain these results, we consider a difference quotient and first show that it is bounded by using the H\"older results obtained in \cite{arroyobp}, and then that the quotient satisfies suitable extremal inequalities.
Then we use H\"older results from \cite{arroyobp} on the quotient, and iterate to obtain the $C^{1,\gamma}$-estimate.
This sketch is along the lines of the PDE proof in Section 5.3 of \cite{caffarellic95}, see also \cite{caffarellis09} and \cite{caffarellitu20}. However, there is a substantial difficulty: it is well known that value functions we consider are not even continuous (see Section \ref{sec:dis} below). The regularity techniques in PDEs use heuristically speaking the fact that there is information available in all scales. Concretely, scaling arguments in the space variable are freely used in those proofs. 
For a discrete process, the step size $\eps$ sets a natural limit for the scale, and this limitation has some crucial effects, for example the mentioned discontinuity. Thus even a meaningful statement of $C^{1,\gamma}$-estimate is not immediate in this context, and $\eps$ needs to be carefully taken into account in the proofs. The $C^{1,\gamma}$-estimate 
\[
\left|u(x)-2u\left(\frac{x+y}{2}\right)+u(y)\right|\leq C\Big(|x-y|^{1+\gamma}+\eps^{1+\gamma}\Big),
\]
and the other estimate
we obtain in Theorem \ref{teo.C1a}
 are and need to be necessarily asymptotic
 (notice the extra term $\eps^{1+\gamma}$ that appears in the right hand side). To the best of our knowledge, only H\"older and Lipschitz estimates, for example in \cite{peress08, luirops13,luirop18} as well as in \cite{arroyolpr20} (under some additional assumptions, and different operators), have been so far available in this context. 
 
The $C^{1,\gamma}$-estimates in Theorem \ref{teo.C1a} suffice in many applications, for example: in passing to the limit with value functions and discrete gradients, and then obtaining the $C^{1,\gamma}$-estimate for the corresponding limiting PDE. We discuss this in Section \ref{sec:convergence}.
As examples of the limit PDEs that can appear in our context, let us mention
 $$ 
 \lambda_k (D^2u) + \Delta u = f, 
 $$
 where 
 $$
 \lambda_k (D^2 u) = \inf_{dim (V) = k} \sup_{v \in V} \langle D^2u \, v; v \rangle
 $$
 is the $k-$th eigenvalue of $D^2u$, and $V$ is a subspace of a specified dimension, see also \cite{blancr19}. The applicability of the results is by no means limited to this example, but rather apply to many kind of fully nonlinear uniformly elliptic PDEs.

The H\"older regularity with mild assumptions  has been crucial in obtaining the higher regularity results in the past: De Giorgi and Nash established such an estimate for divergence form equations and thus solved the final step in Hilbert's 19th problem. In 1979 Krylov and Safonov \cite{krylovs79} established H\"older regularity for second order elliptic equations in nondivergence form with merely bounded and measurable coefficients.  This opened a way to develop higher regularity theory for fully nonlinear elliptic equations. In the case of the viscosity solutions  in this context $C^{1,\gamma}$-regularity among other results can be found in \cite{caffarelli89}.  For an account and further references about this development, see for example \cite{gilbargt01,caffarellic95}.
The results in \cite{arroyobp, arroyobp2} give us a counterpart for Krylov-Safonov regularity results for the discrete stochastic processes we are studying here, and are a key to the higher regularity established in this paper. 
The solution $u$ may not be continuous, but the size of the jumps can be controlled in terms of $\eps$. 
In \cite{arroyobp} an asymptotic $\gamma$-H\"older regularity result was obtained, that is 
\[
|u(x)-u(z)|\leq C\Big(|x-z|^\gamma+\eps^\gamma\Big),
\]
see Theorem~\ref{HolderL+L-} for details.

As an example that is not covered by our results, we mention
the tug-of-war with noise. In this case the simplest version of the dynamic programming principle with the running cost reads as 
\begin{align*}
u(x)=\frac{\alpha}{2}\left( \sup_{B_\varepsilon(x)} u + \inf_{B_ \varepsilon(x) } u\right)+\beta \vint_{B_\varepsilon(x)} u(z) dz + \varepsilon^2 f(x),
\end{align*}
where $\alpha+\beta=1,\ \alpha,\beta\ge 0$, see \cite{manfredipr12}. The H\"older estimates in \cite[Section 7]{arroyobp} hold for this example: the key difference is that now extremal inequalities are needed for a difference quotient, cf.\ (\ref{eq:quotient-super-sub}), and for the tug-of-war with noise such inequalities do not hold for the difference.
The  $C^{1,\gamma}$-regularity for the tug-of-war game with noise is a well known open problem in the field; a related problem for pure tug-of-war is mentioned in \cite{peresssw09}. In the PDE context, tug-of-war with noise is related to the $p$-Laplace equation,
$\Delta_p u = \mbox{div} (|\nabla u |^{p-2} \nabla u) =f$.

\section{Preliminaries}

Let us start with the dynamic programming principle given as an example in the introduction
\[
u(x)=\alpha \sup_{\nu\in B_\Lambda} \frac{u(x+\eps \nu)+u(x-\eps \nu)}{2}
+\beta\vint_{B_1} u(x+\eps y)\,dy
\]
where $\beta,\eps,\Lambda>0$, $\alpha\geq 0$, $\alpha+\beta=1$, and $\vint_{B_1}:=\abs{B_1}^{-1}\int_{B_1}$ denotes the integral average.
We can rewrite this as $$\L^+_\eps u=0$$ where $\L^+_\eps u$ is as in the next definition, and represents an extremal operator analogous to the extremal Pucci operator in the PDE theory.

\begin{definition}
Let $u:\R^n\to\R$ be a bounded Borel measurable function. We define the extremal operator
\begin{equation*}
\begin{split}
	\L_\eps^+ u(x)
	=
	~&
	\frac{1}{2\eps^2}\bigg(\alpha \sup_{\nu\in B_\Lambda} \delta u(x,\eps \nu) +\beta\vint_{B_1} \delta u(x,\eps y)\,dy\bigg),
\end{split}
\end{equation*}
where $\delta u(x,\eps y)=u(x+\eps y)+u(x-\eps y)-2u(x)$. Operator $\L_\eps^-$ is defined analogously just replacing $\sup$ by $\inf$.
\end{definition}

By using the above extremal operators, we specify the admissible operators we consider in this paper. 
\begin{definition}[Admissible operator]
\label{def:our-operator}
We consider operators $\L_\eps$ defined on bounded Borel functions such that
\begin{enumerate} [label=(H\arabic*),ref=(H\arabic*)]
\item
\label{h1}
$\L_\eps$ is uniformly elliptic, i.e.\
\[
\L_\eps^- v\leq \L_\eps(u+v)-\L_\eps  u\leq \L_\eps^+  v,
\]
\item
\label{h2}
$\L_\eps$ is independent of $x$, that is for $u_y(x)=u(x+y)$,
\[
\L_\eps u=0 \quad \text{implies} \quad \L_\eps u_y=0. 
\]
\end{enumerate}
\end{definition}

For example, $\L_\eps^-$, $\L_\eps^+$ and 
\[
\L_\eps u(x)=\frac{1}{2\eps^2}\left(\alpha\delta u(x,\eps \nu)+\beta \vint_{B_1} \delta u(x,\eps z) \,dz\right)
\]
where $\nu\in B_\Lambda$ is a fixed vector, are admissible.
We can also consider
\[
\L_\eps u(x)=\frac{1}{2\eps^2}\left(\alpha\sup_{\nu \in V}\delta u(x,\eps \nu)+\beta \vint_{B_1} \delta u(x,\eps z) \,dz\right)
\]
where $V\subset B_\Lambda$.
More generally, we can also consider Isaacs type operators
\[
\L_\eps u(x)=\frac{1}{2\eps^2}\left(\alpha\sup_{V \in \mathcal V}\inf_{\nu \in V}\delta u(x,\eps \nu)+\beta \vint_{B_1} \delta u(x,\eps z) \,dz\right)
\]
where  $\mathcal V \subset\mathcal P(B_\Lambda)$, a subset of the power set.
We can also consider this operator interchanging the order of $\sup$ and $\inf$.

\section{Asymptotic $C^{1,\gamma}$-regularity}

Our starting point is the asymptotic $\gamma$-H\"older theorem from \cite{arroyobp}, see also \cite{arroyobp2}. This is in a sense the Krylov-Safonov \cite{krylovs79} regularity estimate for our discrete operators. 
\begin{theorem}[\cite{arroyobp}]
\label{HolderL+L-}
There exists $\eps_0>0$ such that if a bounded Borel function $u$ satisfies 
\begin{equation}
\label{L+L-ineq}
\L_\eps^+u\ge -\rho \quad \mbox{ and } \quad \L_\eps^-u\le  \rho, 
\end{equation}
in $B_{2R}$ where $\eps<\eps_0R$ and $\rho\geq 0$, then there exist $\gamma>0$ and $C>0$ such that
\begin{equation}
\label{asymHolder}
|u(x)-u(z)|\leq \frac{C}{R^\gamma}\left(\|{u}\|_{L^{\infty}(B_{2R})}+R^2\rho\right)\Big(|x-z|^\gamma+\eps^\gamma\Big)
\end{equation}
for every $x, z\in B_R$.
\end{theorem}

Above $\eps_0$, $\gamma$ and $C$ only depend on the values $\Lambda$, $\alpha$, $\beta$, and the dimension $n$ (that appear in the definition of $\L_\eps^+$), but not on $\eps$. The precise value of $C$ might vary as usual. Here we also assume $\Lambda>0$ without loss of generality.

Recall that our standing assumption is that $\L_\eps$ is admissible (uniformly elliptic and translation invariant).
The H\"older result implies that solutions to 
$$\L_\eps u=0,$$ 
are locally asymptotically $\gamma$-H\"older. 
Our first goal is to prove that the exponent $\gamma$ can be taken to be 1, that is, the solutions are locally asymptotically Lipschitz. 

\begin{theorem}
\label{Lipschitz}
There exist $\eps_0,\theta,C>0$ such that if a bounded Borel function $u$ satisfies  $\L_\eps u=0$ in $B_R$ with $\eps<R\eps_0$, then 
\begin{equation}
\label{LipEq}
|u(x)-u(z)|\leq \frac{C}{R} \|{u}\|_{L^{\infty}(B_{2R})} \big(|x-z|+\eps\big)
\end{equation}
for every $x, z\in B_{\theta R}$.
\end{theorem}

The strategy that we use to prove our results can be described as follows: In Section 5.3 from \cite{caffarellic95}, $C^{1,\gamma}$-regularity is proved for solutions of certain PDEs.
We follow a similar path, but we have to include an extra $\eps$ term.
This is necessary, as pointed out in the introduction,  since the solutions to $\L_\eps u=0$ need not to be even continuous.  The idea is  to consider the difference quotient (\ref{eq:diff-quot}) below, and first show that it is bounded by using Theorem~\ref{HolderL+L-} on $u$, and that the quotient satisfies suitable extremal inequalities. Then we can use Theorem~\ref{HolderL+L-} on the quotient, and obtain the Lipschitz estimate Theorem~\ref{Lipschitz} by iterating the resulting estimate. Finally, doing one more iteration in  Theorem \ref{teo.C1a}, we obtain the $C^{1,\gamma}$-estimate.

Given $u:\R^n\to \R$ and $e$, a unit vector, we define
\begin{align}
\label{eq:diff-quot}
u_h^\gamma(x)=R^\gamma\frac{u(x+eh)-u(x)}{|h|^\gamma+\eps^\gamma}
\end{align}
for $h\in \R$.
Observe that if $\L_\eps u=0$, then by Theorem~\ref{HolderL+L-} we have that $u_h^\gamma$ is bounded, since by this theorem 
\begin{align*}
\abs{u(x+eh)-u(x)}\le \frac{C}{R^\gamma}\|{u}\|_{L^\infty(B_{2R})} (|h|^{\gamma}+\eps^{\gamma})
\end{align*}
for $x\in B_{R/2}$ and $|h|<R/2$, 
and thus it holds that
\begin{equation}
\label{uh<u:bound}
\|u_h^\gamma\|_{L^\infty(B_{R/2})}\le C\|{u}\|_{L^\infty(B_{2R})}.
\end{equation}

Now, we have that \ref{h1} implies that if $\L_\eps u \geq 0$ and $\L_\eps v \leq 0$ then 
\[
\L_\eps^+(u-v)\geq 0.
\]
In fact, we have
$
0\leq \L_\eps u \leq \L_\eps v +\L_\eps^+(u-v)\leq \L_\eps^+(u-v).
$
Analogously, if $\L_\eps u \leq 0$ and $\L_\eps v \geq 0,$
then 
\begin{align*}
\L_\eps^-(u-v)\leq 0.
\end{align*}
So if $\L_\eps u=0$ then by hypothesis \ref{h1}, \ref{h2} and the one homogeneity of $\L_\eps^+$ and $\L_\eps^-$, we get 
\begin{align}
\label{eq:quotient-super-sub}
\L_\eps^- u_h^\gamma \leq 0\leq \L_\eps^+ u_h^\gamma.
\end{align}
That is, $u_h^\gamma$ is under the hypothesis of Theorem~\ref{HolderL+L-}.
So now our goal is to translate the resulting H\"older estimate for $u_h^\gamma$ combined with \eqref{uh<u:bound} into a result for $u$.
We do that in the following lemma, Lemma~\ref{lem:iter}.

Observe that if we consider $\tilde u$ given by $\tilde u(x)=u(xR)$ and $\tilde \eps=\eps / R$, then $\L^+_{\tilde \eps} \tilde u(x)=R^2\L^+_\eps u(Rx)$.
In fact, we have
$$
\delta \tilde{u}(x, \tilde{\eps} y) = \tilde{u}(x+
\tilde{\eps} y)+\tilde{u}(x- \tilde{\eps} y)-2 \tilde{u}(x)
= u( R x+
\eps y)+ u( R x- \eps y)-2 u (Rx)
$$
and hence
\begin{equation}
\label{scaleL+}
\begin{array}{l}
\displaystyle 
\L^+_{\tilde \eps} \tilde u (x) = \frac{1}{2 \tilde{\eps}^2}\bigg(\alpha \sup_{\nu\in B_\Lambda} \delta \tilde{u}(x, \tilde{\eps} \nu) +\beta\vint_{B_1} \delta \tilde{u}(x, \tilde{\eps} y)\,dy\bigg)
\\[10pt]
\ \displaystyle =  \frac{R^2}{2 \eps^2}\bigg(\alpha \sup_{\nu\in B_\Lambda} \delta u ( R x, \eps \nu) +\beta\vint_{B_1} \delta u( R x, \eps y)\,dy\bigg) = R^2
\L^+_{\eps} u (R x).
\end{array}
\end{equation}
So, $\L_\eps^+ u\geq 0$ is equivalent to $\L_{\tilde\eps}^+ \tilde u\geq 0$.
Also, if we consider with slight a abuse of notation
\begin{align}
\label{eq:difference-quotient-withoutR}
\tilde u_{\tilde h}^\gamma(x)=\frac{\tilde u(x+e \tilde h)-\tilde u(x)}{|\tilde h|^\gamma+\tilde \eps^\gamma},
\end{align}
where $\tilde h=h / R$, by the 1-homogeneity of $\L_\eps^+$ combined with a computation as in \eqref{scaleL+}, we get $\L_{\tilde\eps}^+(\tilde u_{\tilde h}^\gamma)=R^2\L_\eps^+(u_h^\gamma)$.
Therefore $\L_\eps^+(u_h^\gamma) \geq 0$ is equivalent to $\L_{\tilde\eps}^+(\tilde u_{\tilde h}^\gamma) \geq 0$.

These are the inequalities that we need in Lemma~\ref{lem:iter}, Theorem~\ref{Lipschitz} and Theorem~\ref{teo.C1a} in order to apply Theorem~\ref{HolderL+L-}.
Also \eqref{LipEq}, \eqref{hip33}, \eqref{con33}, \eqref{con34a} and \eqref{con34b} scale well with $R$.
So, both the hypothesis and conclusions of the results just mentioned are equivalent when applying them to $u$ or to $\tilde u$ (in a corresponding scaled ball).
Therefore, it is enough to prove these results for $R=1$.

\begin{lemma}
\label{lem:iter}
Let $e\in S^1$ and $u$ be defined in $B_{2R}$ such that 
\[
u_h^\sigma(x)=R^\sigma\frac{u(x+he)-u(x)}{|h|^\sigma+\eps^\sigma},
\]
that is well defined for $x\in B_{R/2}$ and $|h|<R/2$, is bounded. 
If there exists $\tau$ with $\sigma+\tau \neq 1$ such that
\begin{equation}
\label{hip33}
|u_h^\sigma(x+es)-u_h^\sigma(x)|
\leq C {\|{u}\|_{L^{\infty}(B_{2R})}} \left(\frac{|s|^\tau}{R^\tau}+\frac{\eps^\tau}{R^\tau}\right)
\end{equation}
for every $x\in B_{R/8}$ and $|s|<R/8$, then
\begin{equation}
\label{con33}
|u(x+he)-u(x)|
\leq {C\|{u}\|_{L^{\infty}(B_{2R})}}\left(\frac{|h|^\xi}{R^\xi}+ \frac{\eps^{\sigma+\tau}}{R^{\sigma+\tau}}\right)
\end{equation}
for every $x\in B_{R/8}$ and $|h|<R/8$ where $\xi=\min\{1,\sigma+\tau\}$. 
\end{lemma}

\begin{proof}
By considering $\tilde u$ given by $\tilde u(x)=u(xR)$ and $\tilde \eps= \eps / R$ we can assume without loss of generality that $R=1$. 
We define
\[
w(r):=u(x+er)-u(x).
\]
By the definition of $u_{r/2}^\sigma$ and the assumption, we have
\[
\begin{split}
|w(r)-2w(r/2)|
&=|u(x+er)-2u(x+er/2)+u(x)|\\
&=(|r/2|^\sigma+\eps^\sigma)|u_{r/2}^\sigma(x+er/2)-u_{r/2}^\sigma(x)|\\
& \leq (|r/2|^\sigma+\eps^\sigma)C\|{u}\|_{L^{\infty}(B_{2R})}(|r/2|^\tau+\eps^\tau)\\
& \leq C4\|{u}\|_{L^{\infty}(B_{2R})}(|r/2|^{\sigma+\tau}+\eps^{\sigma+\tau})
\end{split}
\]
for every $r<1/4$ so that $r/2<1/8$. 
Next we use a standard iteration argument to get the statement in a neat form. Given $|h|<1/8$, we select $i\in \N$ such that $1/8\leq h2^i<1/4$ and we define $r_0=h2^i$.
Now, we apply the previous inequality for $r= r_0,r_0/2,\dots,r_0/2^{i-1}$ and we obtain
\[
\begin{split}
|w(r_0)-2w(r_0/2)|
&\leq  C4 \|{u}\|_{L^{\infty}(B_{2R})} (|r_0|^{\sigma+\tau}+\eps^{\sigma+\tau})\\
|2w(r_0/2)-2^2w(r_0/2^2)|
&\leq  2C4 \|{u}\|_{L^{\infty}(B_{2R})} (|r_0/2|^{\sigma+\tau}+\eps^{\sigma+\tau})\\
&\vdots\\
|2^{i-1}w(r_0/2^{i-1})-2^iw(r_0/2^{i})|
&\leq  2^{i-1}C4 \|{u}\|_{L^{\infty}(B_{2R})}(|r_0/2^{i-1}|^{\sigma+\tau}+\eps^{\sigma+\tau}).
\end{split}
\]
Adding these inequalities we get
\[
|w(r_0)-2^iw(r_0/2^i)|
\leq C4 {\|{u}\|_{L^{\infty}(B_{2R})}}\!\!\left(|r_0|^{\sigma+\tau}\sum_{j=0}^{i-1}2^{j(1-(\sigma+\tau))}
+2^i\eps^{\sigma+\tau}\right)\!\!.
\]
Then, recalling that $r_0/2^i=h$ and $1/8\leq h2^i$, we have
\[
\begin{split}
&|w(h)|
\leq 2^{-i}|w(r_0)|+2^{-i}C4 {\|{u}\|_{L^{\infty}(B_{2R})}}
\!\! \left(|r_0|^{\sigma+\tau}\sum_{j=0}^{i-1}2^{j(1-(\sigma+\tau))}
+2^i\eps^{\sigma+\tau}\right)\\
&\leq
8|h|2\|{u}\|_{L^{\infty}(B_{2R})}+C4 {\|{u}\|_{L^{\infty}(B_{2R})}}
\!\!\left(|r_0|^{\sigma+\tau}2^{-i}\sum_{j=0}^{i-1}2^{j(1-(\sigma+\tau))}
+\eps^{\sigma+\tau}\right)\!\! .
\end{split}
\]

If $\sigma+\tau<1$, we bound the sum by $\frac{2^{i(1-(\sigma+\tau))}}{2^{1-(\sigma+\tau)}-1}$ and we get
\[
\begin{split}
&|w(h)|
\\&\leq
16|h|\|{u}\|_{L^{\infty}(B_{2R})}+ \frac{C4{\|{u}\|_{L^{\infty}(B_{2R})}}}{2^{1-(\sigma+\tau)}-1} \left(|r_0|^{\sigma+\tau}2^{-i}2^{i(1-(\sigma+\tau))}
+\eps^{\sigma+\tau}\right)\\
&=
16|h|\|{u}\|_{L^{\infty}(B_{2R})}+\frac{C4}{2^{1-(\sigma+\tau)}-1} {\|{u}\|_{L^{\infty}(B_{2R})}}\left(|h|^{\sigma+\tau}+\eps^{\sigma+\tau}\right)
\end{split}
\]
and the result follows. 

If $\sigma+\tau>1$ we bound the sum by $\frac{1}{1-2^{1-(\sigma+\tau)}}$ and we get
\[
\begin{split}
|w(h)|
&\leq
16|h|\|{u}\|_{L^{\infty}(B_{2R})}+ \frac{C4}{1-2^{1-(\sigma+\tau)}} {\|{u}\|_{L^{\infty}(B_{2R})}}\left(|r_0|^{\sigma+\tau}2^{-i}
+\eps^{\sigma+\tau}\right)\\
&=
16|h|\|{u}\|_{L^{\infty}(B_{2R})}+\frac{C4}{1-2^{1-(\sigma+\tau)}} {\|{u}\|_{L^{\infty}(B_{2R})}}\left(|h||r_0|^{\sigma+\tau-1}+\eps^{\sigma+\tau}\right),
\end{split}
\]
we bound $r_0<1/4$ and the result follows. 
\end{proof}

Now, by iterating this result, we prove Theorem~\ref{Lipschitz}.

\begin{proof}[Proof of Theorem~\ref{Lipschitz}] 
By considering $\tilde u$ given by $\tilde u(x)=u(xR)$ and $\tilde \eps={\eps} /R$ we get $\L_{\tilde \eps} \tilde u=0$, therefore we can assume again without loss of generality that $R=1$.

Observe that the result follows from a bound for 
\[
u_h^1(x)=
\frac{u(x+eh)-u(x)}{|h|+\eps}
\]
for $x\in B_\theta$, $e\in S^1$ and $|h|<\theta$. 
We have shown in Lemma~\ref{lem:iter} that an asymptotic H\"older result for $u_h^\sigma$ allows us to get an improved H\"older result for $u$.
Then, we can consider $u_h^\sigma$ for a larger $\sigma$ and then repeat the idea to reach the exponent 1.

Recalling Theorem~\ref{HolderL+L-}, for $R=1/2$, we have 
\[
|u(x+he)-u(x)|
\leq C\|{u}\|_{L^\infty(B_{1})}(|h|^\gamma+\eps^{\gamma})
\]
for $|x|,|h|<1/4$.
Therefore,
\[
u_h^\gamma(x)=\frac{u(x+he)-u(x)}{|h|^\gamma+\eps^\gamma}
\]
is bounded by $C \|{u}\|_{L^\infty(B_1)}$.
As in (\ref{eq:quotient-super-sub}), we have $\L_\eps^- u_h^\gamma \leq 0\leq \L_\eps^+ u_h^\gamma$, and  by Theorem~\ref{HolderL+L-}, for $R=1/8$, we get
\[
|u_h^\gamma(x+he)-u_h^\gamma(x)|
\leq C \|{u}\|_{L^\infty(B_{1})} (|h|^\gamma+\eps^{\gamma})
\]
for $|x|,|h|<1/16$.
Therefore, using again Lemma~\ref{lem:iter}, we get
\[
\begin{split}
|u(x+he)-u(x)|
&\leq C \|{u}\|_{L^\infty(B_{1})} (|h|^\xi+ \eps^{2\gamma})
\end{split}
\]
for $|x|,|h|<1/16$, where $\xi=\min\{1,2\gamma\}$. 

If $2\gamma<1$, we now consider 
\[
u_h^{2\gamma}(x)=\frac{u(x+he)-u(x)}{|h|^{2\gamma}+\eps^{2\gamma}}.
\]
As before we have $\L_\eps^- u_h^{2\gamma} \leq 0\leq \L_\eps^+ u_h^{2\gamma}$ and therefore Theorem~\ref{HolderL+L-} applies. Hence, we get
\[
\begin{split}
|u(x+he)-u(x)|
&\leq C \|{u}\|_{L^{\infty}(B_1)}(|h|^\xi+ \eps^{3\gamma})
\end{split}
\]
for $|x|,|h|<1/4^3$. where $\xi=\min\{1,3\gamma\}$. 

Observe that Theorem~\ref{HolderL+L-} gives us a value of $\gamma$ but it holds for every smaller $\gamma$. 
So, we can consider $\gamma$ such that $k\gamma\neq 1$ for every $k\in\N$.
We iterate this procedure until we reach $k\gamma>1$, and get
\[
\begin{split}
|u(x+he)-u(x)|
&\leq C \|{u}\|_{L^{\infty}(B_1)}(|h|+ \eps^{k\gamma})
\end{split}
\]
for $|x|, |h|<1/4^k$. 
We can take $\eps_0<1$ so that $\eps^{k\gamma}<\eps$ and the result follows for $\theta=\frac{1}{4^k2}$.
\end{proof}

To obtain the desired $C^{1,\gamma}$ type asymptotic estimates,  we utilize the difference quotient
\[
u_h^1(x)=\frac{u(x+he)-u(x)}{|h|+\eps},
\]
and use  Theorem \ref{HolderL+L-} once more along with the previous theorem. 
\begin{theorem}
\label{teo.C1a}
There exist $\eps_0, \kappa>0$ and $C>0$ such that if $u$ is a bounded Borel function such that $\L_\eps u=0$ in $B_{R}$ with $\eps<\eps_0$, it holds that
\begin{equation}
\label{con34a}
|u_h^1(x)-u_h^1(y)|\leq \frac{C}{R^{\gamma}}C\|{u}\|_{L^{\infty}(B_1)} \Big(|x-y|^\gamma+\eps^\gamma\Big) 
\end{equation}
for $|h|<\kappa R$, 
and
\begin{equation}
\label{con34b}
\left|u(x)-2u\left(\frac{x+y}{2}\right)+u(y)\right|\leq \frac{C}{R^{1+\gamma}}\|{u}\|_{L^{\infty}(B_{R})}\Big(|x-y|^{1+\gamma}+\eps^{1+\gamma}\Big)
\end{equation}
for every $x,y\in B_{\kappa R}$.
\end{theorem}

\begin{proof} 
Again we can consider $\tilde u$ given by $\tilde u(x)=u(xR)$ and $\tilde \eps={\eps} /R$ to assume without loss of generality that $R=1$.

For $u_h^1(x)$, we obtain from Theorem~\ref{Lipschitz} that it is bounded in $B_\theta$ by $C\|{u}\|_{L^{\infty}(B_1)}$, and satisfies extremal inequalities similar as the ones in (\ref{eq:quotient-super-sub}). Hence,
we can apply Theorem \ref{HolderL+L-} to $u_h^1$, and get
\begin{align}
\label{eq:C1a-form2}
|u_h^1(z)-u_h^1(x)|\leq C\|{u}\|_{L^{\infty}(B_1)} \Big(|z-x|^\gamma+\eps^\gamma\Big) \end{align}
in $B_{\theta/4}$, proving the first inequality.

Next we prove the second inequality. Let $e$ be a unit vector, $h\in\R$ such that $y-x=2eh$, 
so that the claim reads as 
\[
|u(x+2he)-2u(x+he)+u(x)|\leq C\|{u}\|_{L^{\infty}(B_1)}\Big(|h|^{1+\gamma}+\eps^{1+\gamma}\Big).
\]
But now this follows from (\ref{eq:C1a-form2}), selecting $z$ as $x+he$.
\end{proof}

Later in Section \ref{sec:convergence}, we will show that the estimate in addition to being interesting on its own right, is sufficient in order to pass to a limit as $\eps \to 0$. Thus, we obtain as an application a $C^{1,\gamma}$-regularity for the limit function (that turns out to be a solution to a PDE).

\subsection{Dynamic programming principle with $x$ dependent RHS}

In this section we present a version with an $x$-dependent right hand side $f$. In terms of stochastic processes, $f$ represents a running payoff that can change according to the spatial location.

We consider solutions to the equation 
$$
\L_\eps u=f,
$$ 
where $f$ is a bounded Lipschitz function. As in the previous section, the idea is  to consider the difference quotient, to show that it is bounded and that it satisfies suitable extremal inequalities. Then, we can use Theorem~\ref{HolderL+L-} on the quotient, and obtain the Lipschitz estimate (Theorem~\ref{Lipschitz.f} below), which is a counterpart to Theorem~\ref{Lipschitz}. 

To be more precise, as before we use
\[
u_h^\gamma(x)=R^\gamma\frac{u(x+eh)-u(x)}{|h|^\gamma+\eps^\gamma},
\]
and our aim is to show that this quantity is bounded.
Observe that since  
\[
\L_\eps^+ u\geq -\|{f}\|_{L^{\infty}}
\quad
\mbox{ and }
\quad
\L_\eps^- u\leq \|{f}\|_{L^{\infty}},
\]
 it follows by Theorem~\ref{HolderL+L-} that 
\begin{align}
\label{eq:Holder_f}
\abs{u(x+eh)-u(x)}\le \frac{C}{R^\gamma}\left(\|{u}\|_{L^\infty(B_{2R})}+R^2
\|{f}\|_{L^\infty(B_{2R})}\right) \big(|h|^{\gamma}+\eps^{\gamma}\big)
\end{align}
for $x\in B_{R/2}$ and $|h|<R/2$.
Thus
\[
\|{u_h^\gamma}\|_{L^\infty(B_{R/2})}\le C\left(\|{u}\|_{L^\infty(B_{2R})}+R^2
\|{f}\|_{L^\infty(B_{2R})}\right),
\]
i.e.\ we obtain the desired boundedness.

Next we show that $u_h^\gamma$ satisfies extremal inequalities. It  holds that 
$$\L_\eps^+(u(x))\ge  f(x),  \quad \mbox{ and } \quad \L_\eps^-(u(x))\le f(x),$$ and thus we have
\begin{align}
\label{eq:extremal}
\L_\eps^+(u_h^\gamma)(x)
&=
\frac{R^\gamma}{|h|^\gamma+\eps^\gamma} \L_\eps^+(u(x+eh)-u(x))\nonumber\\
&\geq
\frac{R^\gamma}{|h|^\gamma+\eps^\gamma} (\L_\eps(u(x+eh))-\L_\eps(u(x)))\nonumber\\
&\geq 
\frac{R^\gamma}{|h|^\gamma+\eps^\gamma} (f(x+eh)-f(x))\nonumber\\
&\geq 
-\frac{R^\gamma}{|h|^\gamma+\eps^\gamma} |h| \Lip(f)\nonumber\\
&\geq 
-R \Lip(f).
\end{align}
Analogously  $\L_\eps^-(u_h^\gamma)\leq R \Lip(f)$.

Recall that if we consider $\tilde u$ given by $\tilde u(x)=u(xR)$ and $\tilde \eps=\eps / R$, then $\L^+_{\tilde \eps} \tilde u(x)=R^2\L^+_\eps u(Rx)$.
If we define $\tilde f(x)=R^2 f(xR)$ we have 
\begin{equation}
\label{scale}
\Lip(\tilde f)=R^3\Lip f
\quad \text{and} \quad 
\| {\tilde f} \|_{L^\infty(B_{2})}=R^2\|{f}\|_{L^\infty(B_{2R})}.
\end{equation}
So, $\L_\eps^+ u\geq -\|{f}\|_{L^{\infty}}$ is equivalent to $\L_{\tilde\eps}^+ \tilde u\geq -\|{\tilde f}\|_{L^{\infty}}$.
Also recall that for $\tilde h=h / R$ we have $\L_{\tilde\eps}^+(\tilde u_{\tilde h}^\gamma)=R^2\L_\eps^+(u_h^\gamma)$, where $\tilde u_{\tilde h}^\gamma$ was defined in (\ref{eq:difference-quotient-withoutR}) and $u_h^\gamma(x)$ above at the beginning of this section.
Therefore $\L_\eps^+(u_h^\gamma) \geq -R \Lip(f)$ is equivalent to $\L_{\tilde\eps}^+(\tilde u_{\tilde h}^\gamma) \geq -\Lip(\tilde f)$.

These are the inequalities that we will need in Lemma~\ref{lem:iter.f}, Theorem~\ref{Lipschitz.f} and Theorem~\ref{teo.C1a.f} in order to apply Theorem~\ref{HolderL+L-}.
Also \eqref{hip35}, \eqref{con35}, \eqref{con36}, \eqref{con37a} and \eqref{con37b} scale well with $R$ according to \eqref{scale}.
So, both the hypothesis and conclusions of just mentioned results are equivalent when applying them to $u$ or to $\tilde u$ (in a corresponding scaled ball).
Therefore, we can argue as before proving these results for a fixed radius $R=1$.

Next, we state a counterpart to Lemma~\ref{lem:iter}. Its proof follows from exactly the same calculations as in Lemma~\ref{lem:iter} since we have boundedness and suitable extremal inequalities at our disposal.

\begin{lemma}
\label{lem:iter.f}
Let $e\in S^1$ and $u$ be defined in $B_{2R}$ such that 
\[
u_h^\sigma(x)=R^\sigma\frac{u(x+he)-u(x)}{|h|^\sigma+\eps^\sigma},
\]
that is well defined for $x\in B_{R/2}$ and $|h|<R/2$, is bounded.
If there exists $\tau$ with $\sigma+\tau \neq 1$ such that
\begin{equation}
\label{hip35}
\begin{array}{l}
\displaystyle 
|u_h^\sigma(x+es)-u_h^\sigma(x)| \\[10pt]
\quad \displaystyle \leq {C\left(\|{u}\|_{L^\infty(B_{2R})}+R^2\|{f}\|_{L^\infty(B_{2R})}+R^3\Lip f\right) }\left(\frac{|s|^\tau}{R^\tau}+\frac{\eps^\tau}{R^\tau}\right)
\end{array}
\end{equation}
for every $x\in B_{R/8}$ and $|s|<R/8$, then
\begin{equation}
\label{con35}
|u(x+he)-u(x)|
\leq C(u,f)\left(\frac{|h|^\xi}{R^\xi}+ \frac{\eps^{\sigma+\tau}}{R^{\sigma+\tau}}\right)
\end{equation}
for every $x\in B_{R/8}$ and $|h|<R/8$ where $\xi=\min\{1,\sigma+\tau\}$. 
Here   
$$C(u,f)=C\left(\|{u}\|_{L^{\infty}(B_{2R})} +R^2\|{f}\|_{L^\infty(B_{2R})}  + R^3\Lip f\right).$$ 
\end{lemma}

We are ready to state and prove the counterpart of Theorem~\ref{Lipschitz}.

\begin{theorem}
\label{Lipschitz.f}
There exist $\eps_0,\theta,C>0$ such that if a bounded Borel function $u$ satisfies $\L_\eps u=f$ in $B_R$ with $\eps<R\eps_0$, then 
\begin{equation}
\label{con36}
|u(x)-u(z)|\leq \frac{C}{R}\left( \|{u}\|_{L^{\infty}(B_R)}+R^2\|{f}\|_{L^\infty(B_{R})}+R^3\Lip f\right)\big(|x-z|+\eps\big)
\end{equation}
for every $x, z\in B_{\theta R}$.
\end{theorem}

\begin{proof} Again we assume that $R=1$. 
As pointed out in (\ref{eq:Holder_f}), similarly as in the proof of  Theorem~\ref{Lipschitz}, we get  by Theorem~\ref{HolderL+L-}, for $R=1/2$, that 
\[
|u(x+he)-u(x)|
\leq C\left(\|{u}\|_{L^\infty(B_{1})}+\frac{1}{4}\|{f}\|_{L^\infty(B_{1})}\right)(|h|^\gamma+\eps^{\gamma})
\]
for $\abs{x},|h|<1/4$.
Therefore 
\[
u_h^\gamma(x)=\frac{u(x+he)-u(x)}{|h|^\gamma+\eps^\gamma}
\]
is bounded by $C\left(\|{u}\|_{L^\infty(B_{1})}+\frac{1}{4}\|{f}\|_{L^\infty(B_{1})}\right)$.
Then, by Theorem~\ref{HolderL+L-}, for $R=1/8$, we get
\[
|u_h^\gamma(x+he)-u_h^\gamma(x)|
\leq C\left( {\|{u}\|_{L^\infty(B_{1})}+\|{f}\|_{L^\infty(B_{1})}}+\Lip f\right)(|h|^\gamma+\eps^{\gamma})
\]
for $|x|,|h|<1/16$. 
Therefore, by Lemma~\ref{lem:iter.f} for $R=1/2$ we get
\[
\begin{split}
|u(x+he)-u(x)|
&\leq C\left( \|{u}\|_{L^{\infty}(B_1)}+\|{f}\|_{L^\infty(B_{1})}+\Lip f\right)(|h|^\xi+ \eps^{2\gamma})
\end{split}
\]
for $|x|,|h|<1/16$, where $\xi=\min\{1,2\gamma\}$. 

As before if $2\gamma<1$ we iterate this procedure until we reach the desired exponent 1.
\end{proof}

Also the counterpart for Theorem~\ref{teo.C1a}, i.e.\ the $C^{1,\gamma}$-estimates, holds. For the formulation, recall that
\[
u_h^1(x)=\frac{u(x+he)-u(x)}{|h|+\eps}.
\]
\begin{theorem}
\label{teo.C1a.f}
There exist $\eps_0, \kappa>0$ and $C>0$ such that if $u$ is a bounded Borel function such that $\L_\eps u=f$ in $B_{R}$ with $\eps<\eps_0$, it holds that
\begin{equation}
\label{con37a}
|u_h^1(x)-u_h^1(y)|\leq \frac{C}{R^{\gamma}}\|{u}\|_{L^{\infty}(B_1)} \Big(|x-y|^\gamma+\eps^\gamma\Big) 
\end{equation}
for $|h|<\kappa R$, 
and
\begin{equation}
\label{con37b}
\left|u(x)-2u\left(\frac{x+y}{2}\right)+u(y)\right|\leq \frac{C}{R^{1+\gamma}}\|{u}\|_{L^{\infty}(B_{R})}\Big(|x-y|^{1+\gamma}+\eps^{1+\gamma}\Big)
\end{equation}
where 
$$C=C\Big(\|{u}\|_{L^{\infty}(B_1)}+R^2\|{f}\|_{L^\infty(B_{R})}  + R^3\Lip f\Big)$$ 
for every $x,y\in B_{\kappa R}$.
\end{theorem}

\begin{proof}
As before, assuming that $R=1$, the idea is to observe that $u_h^1(x)$
is bounded by $C (\|{u}\|_{L^{\infty}}+\|{f}\|_{L^\infty})$ using Theorem~\ref{Lipschitz.f}.
Then, a similar estimate as in (\ref{eq:extremal}) applied to $u_h^1$ gives  that 
$$\L_\eps^-(u_h^1)\leq  \Lip(f) \quad \mbox{ and }\quad \L_\eps^+(u_h^1)\ge  -\Lip(f).$$ 
Thus, we can apply Theorem \ref{HolderL+L-} to get
\[
|u_h^1(x)-u_h^1(z)|\leq  C\left(\|{u}\|_{L^{\infty}(B_1)}+\|{f}\|_{L^{\infty}(B_1)}+\Lip(f)\right)\Big(|x-z|^\gamma+\eps^\gamma\Big).
\]
This can be rewritten in terms of $u$ as
\[
\begin{array}{l}
\displaystyle 
|u(x+2he)-2u(x+he)+u(x)| \\[10pt]
\qquad \displaystyle \leq C\left(\|{u}\|_{L^{\infty}(B_1)}+\|{f}\|_{L^{\infty}(B_1)}+\Lip(f)\right)\Big(h^\gamma+\eps^\gamma\Big),
\end{array}
\]
finishing the proof.
\end{proof}

\section{Discrete gradient and convergence}
\label{sec:convergence}

As mentioned in the introduction, this article is partly motivated by the DPPs obtained in stochastic game theory.
These formulas play a key role when studying the connection between stochastic games and PDEs as explained for example in \cite{peresssw09, peress08, manfredipr12, blancr19b,lewicka20}.

Let us briefly describe how this connection is established.
A bounded domain $\Omega\subset \R^n$ and a final payoff function $g:\R^n\setminus\Omega\to \R$ are given.
Then, the value function $u^\eps$ for the game or for the possibly controlled stochastic process verifies 
\begin{align*}
\begin{cases}
\L_\eps u^\eps=0,&\text{ in }\Omega, \\
u^\eps=g,& \text{ in }\R^n\setminus\Omega.
\end{cases}
\end{align*}
Under suitable assumptions it can be proven that there exists a function $u$ and a subsequence such that $u^\eps\to u$ uniformly in $\overline{\Omega}$,
and the function $u$ is a solution to a limit local PDE, $\L u=0$.

The nature of the operator $\L_\eps$ and the limit differential operator $\L$
depends on the particular rules of the game.
For example in \cite{brustadlm20}, a controlled process associated to $\L_\eps^+$ is presented and, in the limit, the dominative $p$-Laplacian operator arises. Also many other examples are possible.

To obtain the convergent subsequence the following Arzel\`{a}-Ascoli type lemma from \cite{manfredipr12} is often used.

\begin{lemma}
\label{lemAA}
Let $\{u^\eps : \overline{\Omega} \to \R,\ \eps>0\}$ be a set of functions such that
\begin{enumerate}
\item there exists $C>0$ such that $\abs{u^\eps (x)}<C$ for
    every $\eps>0$ and every $x \in \overline{\Omega}$,
\item given $\eta>0$ there are constants
    $r_0$ and $\eps_0$ such that for every $\eps < \eps_0$
    and any $x, y \in \overline{\Omega}$ with $|x - y | < r_0 $
    it holds
$$
|u^\eps (x) - u^\eps (y)| < \eta.
$$
\end{enumerate}
Then, there exists a function $u: \overline{\Omega} \to \R$ and a subsequence such that
\[
\begin{split}
u^{\eps}\to u \qquad\textrm{ uniformly in}\quad\overline{\Omega},
\end{split}
\]
as $\eps\to 0$.
\end{lemma}

In order to apply the lemma, the regularity required in {\it (2)} is the key point.
In the game setting the regularity is obtained for example by an analysis of the strategies available to the players combined with the regularity of the boundary of the domain. 
In the interior, the regularity can also be obtained by Lemma~\ref{HolderL+L-}, and the existence of the convergent subsequence follows.

Following these ideas, in this section, we consider $u^\eps$ a bounded family of solutions to 
$$\L_\eps u^\eps=0$$ such that $$u^\eps\to u$$ uniformly in $\Omega$.

Our goal now is to use the obtained regularity estimates to pass to the limit with a discrete gradient. We show that the discrete gradients converge uniformly to the gradient of the limit which is H\"older continuos. Thus, this implies $C^{1,\gamma}$-regularity for solutions to the corresponding limiting PDE.

We define a discrete partial derivative
\[
u^\eps_{\eps e}(x)=\frac{u^{\eps}(x+\eps e)-u^{\eps}(x)}{\varepsilon}
\]
and a discrete gradient
\[
\nabla^\eps u^\eps (x)=(u^\eps_{\eps e_1}(x),\ldots,u^\eps_{\eps e_n}(x)),
\]
where $(e_1,..,e_n)$ denotes the canonical basis.
Observe that we use the notation $u^\eps_{\eps e}$ that 
is similar to the one used in the previous section for a different object.

\begin{proposition}
Let  $u^\eps$ be a bounded family of solutions to $\L_\eps u^\eps=0$ such that $u^\eps\to u$ uniformly in $\Omega$.
Then $u\in C^{1,\gamma}(\Omega)$ and
\[
 \nabla^\eps u^\eps \to  \nabla u,
\]
locally uniformly, with the local estimate
\[
|u^{\eps}(y)-u^{\eps}(x)-\nabla^\eps u^{\eps}(x)(y-x)|
\leq  C(|x-y|^{\gamma+1} + \eps)
\]
that implies 
\[
|u (y)-u (x)-\nabla u(x)(y-x)|
\leq  C|x-y|^{\gamma+1} 
\]
for the limit.
\end{proposition}

\begin{proof}
Let $V$ be a domain such that $\overline{V}\subset\Omega$.
By Theorem~\ref{teo.C1a}, we have that $u^\eps_{\eps e}$ is asymptotically H\"older and uniformly bounded. 
This implies that it is asymptotically continuous, that is, the conditions
of the Arzel\`a-Ascoli type lemma (Lemma \ref{lemAA}) are fulfilled.  
Then, applying Lemma \ref{lemAA} for each coordinate, we obtain that there exists a subsequence such that 
$$
\nabla^\eps u^\eps (x) \to D (x)
$$
uniformly in $\overline{V}$ for some function $D:\overline{V}\to\R^n$.

We observe that by passing to the limit in the estimate provided by Theorem~\ref{teo.C1a} we get that components of $D:\overline{V}\to\R^n$ are H\"older continuous.

Now, we want to prove that this function is indeed the gradient of the limit function $u$.
To do that we prove that the discrete gradient provides a linear function that locally approximates $u^\eps$, then we pass to the limit, and conclude that $D$ approximates $u$.
The computation below follows the idea employed to prove that a function with continuous partial derivatives is differentiable.

Given $x,y\in\R^n$, we define 
\[
k_i=\left\lfloor\frac{(y-x)_i}{\eps}\right\rfloor
\]
and
\[
\begin{split}
y^0&=x+\eps(0,\dots,0)=x\\
y^1&=x+\eps(k_1,0,\dots,0)\\
&\vdots\\
y^i&=x+\eps(k_1,\dots,k_i,0\dots,0)\\
&\vdots\\
y^n&=x+\eps(k_1,\dots,k_n).
\end{split}
\]
We have the estimate
\begin{align*}
|&u^\eps(y)-u^\eps(x)-\nabla^\eps u^\eps (x)\cdot(y-x)|\\
&\leq |u^\eps(y^n)-u^\eps(x)-\nabla^\eps u^\eps (x)\cdot (y^n-x)|
+|u^\eps(y)-u^\eps(y^n)| \\
& \quad +|\nabla^\eps u^\eps (x)\cdot (y-y^n)|.
\end{align*}
Since $|y_i-y^n_i|\leq \eps$, we have
\[
|u^\eps(y)-u^\eps(y^n)|+|\nabla^\eps u^\eps (x)\cdot(y-y^n)|\leq C\eps,
\]
and hence, we get
\begin{align*}
|&u^\eps(y)-u^\eps(x)-\nabla^\eps u^\eps (x)\cdot(y-x)|\nonumber\\
&\leq |u^\eps(y^n)-u^\eps(x)-\nabla^\eps u^\eps (x)\cdot (y^n-x)|+C\eps.
\end{align*} 

Next we write
\[
\begin{array}{l}
\displaystyle 
|u^\eps(y^n)-u^\eps(x)-\nabla^\eps u^\eps (x)\cdot(y^n-x)| \\[10pt]
\quad \displaystyle 
\leq \sum_{i=0}^{N-1}  \Big|u^\eps(y^{i+1})-u^\eps(y^i)- u^\eps_{\eps e_{i+1}} (x)(y^n-x)_{i+1} \Big|.
\end{array}
\]

Recalling that $(y^n-x)_{i+1}=\eps k_{i+1}$, we can estimate each term in the previous sum as 
\[
\begin{split}
& |u^\eps(y^{i+1})-u^\eps(y^i)- u^\eps_{\eps e_{i+1}} (x)\eps k_{i+1}| \\
&\leq 
\sum_{j=0}^{k_{i+1}-1}  \Big|u^\eps(y^i+(j+1)\eps e_{i+1})-u^\eps(y^i+j\eps e_{i+1})-u^\eps_{\eps e_{i+1}} (x)\eps  \Big|
\\
&\leq 
\sum_{j=0}^{k_{i+1}-1}\eps \Big|u^\eps_{\eps e_{i+1}}(y^i+j\eps e_{i+1})-u^\eps_{\eps e_{i+1}} (x)\Big|
\\
&\leq 
k_{i+1}\eps C(|x-y|^\gamma+\eps^\gamma)
\\
&\leq 
 C(|x-y|^{\gamma+1}+|x-y|\varepsilon^{\gamma}).
\end{split}
\]
Finally, combining the previous three estimates, we get
\[
\begin{split}
|u^\eps(y)-u^\eps(x)-\nabla^\eps u^\eps (x)\cdot(y-x)|
&\leq  C(|x-y|^{\gamma+1}+|x-y|\epsilon^{\gamma}+\eps)\\
&\leq  C(|x-y|^{\gamma+1}+\eps).
\end{split}
\]
Passing to the limit as $\eps \to 0$, we get
\[
|u(y)-u(x)-D(x)\cdot(y-x)|
\leq  C|x-y|^{\gamma+1}.
\]
Therefore we conclude that $D=\nabla u$ and that $u$ is differentiable.
We have obtained the convergence 
\[
\nabla^\eps u^\eps \to \nabla u,
\]
along a subsequence, locally uniformly.
Finally, from the uniqueness of the gradient, we get the convergence for the whole sequence.
\end{proof}

\section{Discontinuities}
\label{sec:dis}

In the proof of Lemma~\ref{lem:iter} we obtained an improved exponent for $\abs h$ from $\tau$ to $\min\{1,\sigma+\tau\}$ and for the error $\eps$ from $\tau$ to $\sigma+\tau$. 
In other words, the improvement of the exponent of the term $\abs h$ is bounded by 1 but the exponent of $\eps$ can exceed 1.
Actually, by adapting our argument in Lemma~\ref{lem:iter} and Theorem~\ref{Lipschitz} we can show that
\begin{corollary}
\label{cor:smaller-jumps}
There exist $\eps_0,\theta,C>0$ such that if a bounded Borel function $u$ satisfies  $\L_\eps u=0$ in $B_R$ with $\eps<\eps_0$, then for any $a>0$
\[
|u(x)-u(z)|\leq C\|{u}\|_{L^\infty(B_{R})}\Big(|x-z|+\eps^a\Big)
\]
for every $x,z \in B_{\theta R}$.
\end{corollary}
For a proof of this corollary we refer to the end of the proof of Theorem~\ref{Lipschitz}.

Knowing that the value functions can be discontinuous, the above estimate may look counterintuitive. However, the radius $\theta R$ depends on the exponent $a>0$, and we have to iterate more times for a larger $a$.
We conclude that the jumps of the function are smaller if we are far away from the boundary of the domain.

We illustrate this phenomenon with one dimensional examples modified from \cite[Example 2.2]{manfredipr12}.
Let $\Omega=(0,1)$ and $g:\R\setminus(0,1)\to \R$ given by 
\begin{align*}
g(x)=\begin{cases}
0 &x \le 0, \\
 1& x \ge 1.  
\end{cases}
\end{align*}
We consider the solutions to the two different DPPs: 
\begin{equation}
\label{eq:dppinf1d}
u(x)=\frac{u(x+\eps)+u(x-\eps)}{2}
\end{equation}
and
\begin{equation}
\label{eq:dppp1d}
u(x)=\alpha\frac{u(x+\eps)+u(x-\eps)}{2}+\beta\aveint{x-\eps}{x+\eps}u(y)\,dy,
\end{equation}
with $\alpha\ge 0,\beta>0, \alpha+\beta=1$.
The first DPP corresponds to the one dimensional equation for the tug-of-war game, see \cite{peresssw09}, and does not satisfy our assumptions since it is not uniformly elliptic.
The second one corresponds to the one dimensional tug-of-war with noise as well as both $\L^+_\eps u=0$ and $\L^-_\eps u=0$.
In Figure~\ref{fig:stair} and Figure~\ref{fig:smooth} we present the solutions to these DPPs.

\begin{figure}
\includegraphics[scale=0.5]{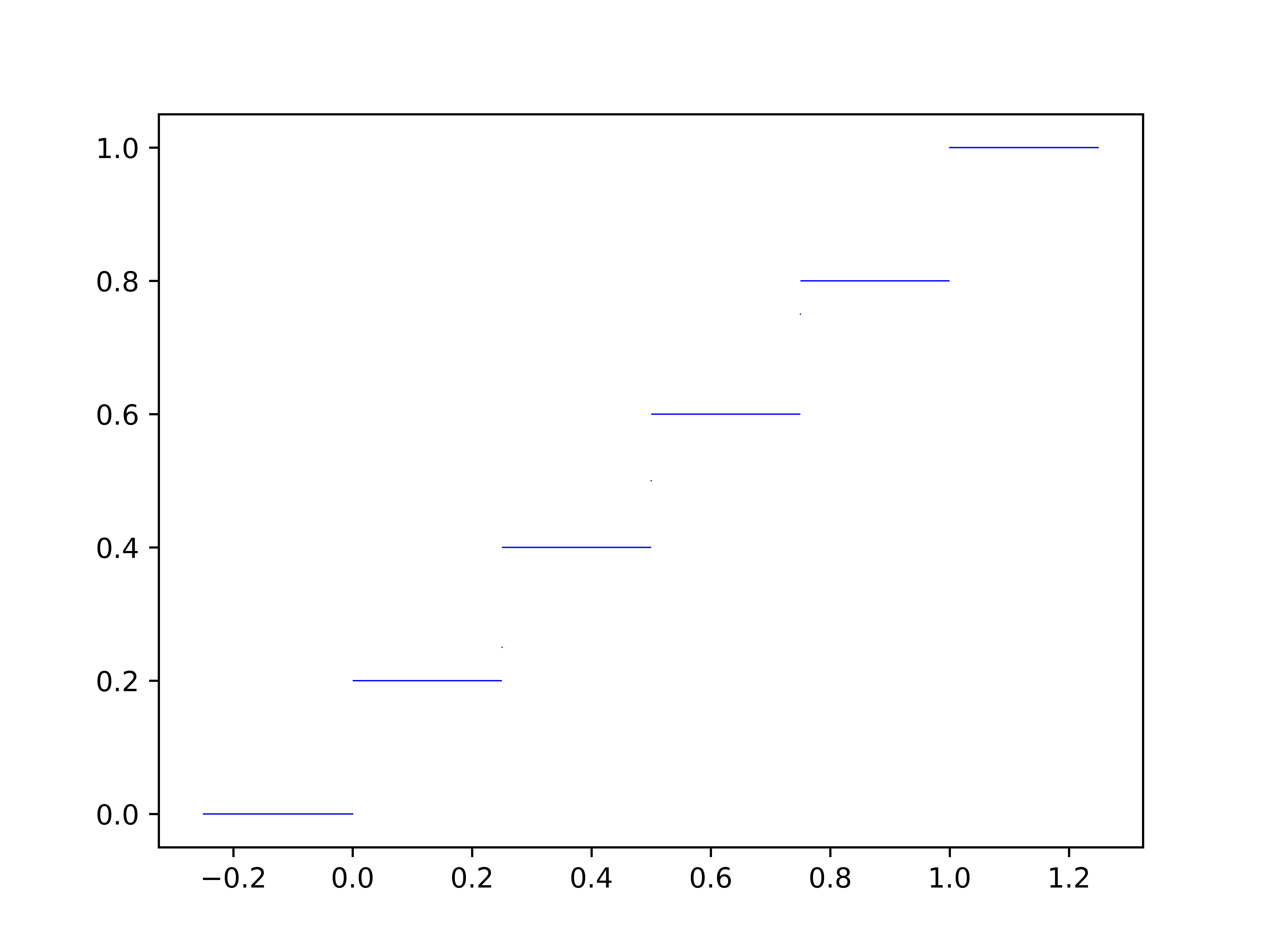}
\caption{The solution to the DPP \eqref{eq:dppinf1d} for $\eps=1/5$.}
\label{fig:stair}
\end{figure}

\begin{figure}
\includegraphics[scale=0.5]{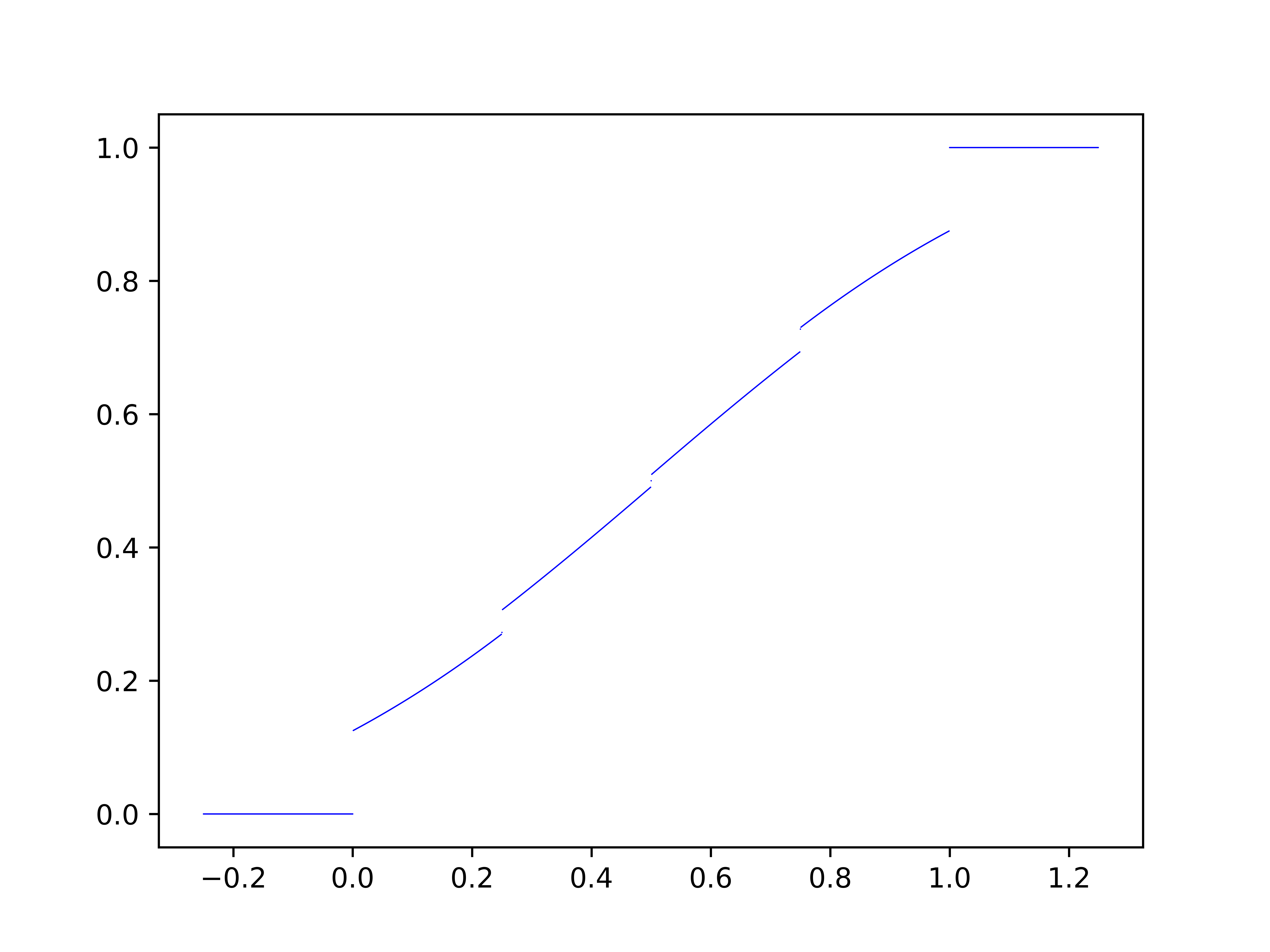}
\caption{The solution to the DPP \eqref{eq:dppp1d} for $\eps=1/5$ and ${\alpha}=\beta=1/2$.}
\label{fig:smooth}
\end{figure}
  
In Figure~\ref{fig:stair} we can observe that all the discontinuities, $\limsup_{y\to x}u^{\eps}(y)-\liminf_{y\to x}u^{\eps}(y)$, are of size $\eps$, meanwhile in Figure~\ref{fig:smooth} the jumps get smaller as we move away from the boundary of the domain.

This phenomenon holds also in general.
Let $\Omega\subset\R^n$ be a bounded domain, $g:\R^n\setminus\Omega\to \R^n$ a bounded function and we consider the solution to the DPP given by $\L^+_\eps u=0$, that is
\begin{equation}
\label{DPPL+1}
u(x)=\alpha \sup_{\nu\in B_1} \frac{u(x+\eps \nu)+u(x-\eps \nu)}{2}
+\beta\vint_{B_1} u(x+\eps y)\,dy, \qquad x \in \Omega, 
\end{equation}
with
$$
u(x) = g(x), \qquad \qquad x \in \mathbb{R}^n \setminus \Omega.
$$
We have taken $\Lambda=1$ to simplify the formulas but the result is valid in general.
Since $\|u\|_{L^{\infty}}\leq \|g\|_{L^{\infty}}$ we know that the jumps in the discontinuities are bounded by $2\|g\|_{L^{\infty}}$.
Inside the domain, if we look at the RHS of \eqref{DPPL+1} we have that $\beta \vint_{B_\eps(x)}u^{\eps}(y)\,dy$ is continuous so the jumps will be at most of size $2\alpha\|g\|_{L^{\infty}}$.

We can iterate this argument.
Let
\[
\Omega_\delta=\{x\in\Omega:\dist(x,\partial\Omega)>\delta\}.
\]
Then, in $\Omega_\eps$ the jumps will be at most of size $2\alpha^2\|g\|_{L^{\infty}}$.
Therefore, in general we get that given $x\in\Omega$, the jump of a possible discontinuity at that point is at most of size
\[
2 \alpha^{\left\lceil\frac{\dist(x,\partial\Omega)}{\eps}\right\rceil} \|g\|_{L^{\infty}}
\]
that goes to zero exponentially fast as $\eps \to 0$ (since $\alpha <1$), and thus this error is consistent with our result.

Notice also that when $\alpha =0$ then the solution to the DPP is continuous inside $\Omega$
but may have a discontinuity (of size $\varepsilon$) on the boundary $\partial \Omega$.
We can also consider solutions to the DPP given by the tug-of-war with noise introduced in \cite{manfredipr12}, that is
\begin{align}
\label{eq:dppp}
u^{\eps}(x)=
\frac{\alpha}{2}\left(\sup_{B_\eps(x)} u^{\eps}+\inf_{B_\eps(x)}u^{\eps}\right)+\beta\vint_{B_\eps(x)}u^{\eps}(y)\,dy.
\end{align}
The same argument that controls the size of the discontinuities also applies to solutions of this DPP.

 Since the observations above might be of independent interest, we state them as a proposition.
\begin{proposition}
\label{prop:sizejumps}
Let $u^{\eps}$ be a solution to \eqref{DPPL+1} or \eqref{eq:dppp}. Then for $x\in \Om\subset \R^n$, it holds that the jumps of $u^{\eps}$ can be at the most of size
\begin{align*}
2\|g\|_{L^{\infty}} \alpha^{\left\lceil\frac{\dist(x,\partial\Omega)}{\eps}\right\rceil}.
\end{align*}
\end{proposition}


\newpage

On behalf of all authors, the corresponding author states that there is no conflict of interest.



\def\cprime{$'$} \def\cprime{$'$} \def\cprime{$'$}

\end{document}